\documentclass[11pt]{article}

\usepackage{amssymb,amsmath,epsfig,amsthm}
\usepackage{mathtools}
\usepackage{graphicx}
\usepackage{color}
\usepackage{fullpage}
\usepackage{tikz}
\usetikzlibrary{positioning,decorations.pathreplacing,matrix,arrows,calc,shapes,arrows}

\numberwithin{equation}{section}
\newtheorem{theorem}{Theorem}[section]
\newtheorem{conjecture}{Conjecture}
\newtheorem{lemma}{Lemma}[section]
\newtheorem{corollary}{Corollary}[section]

\theoremstyle{definition}

\newtheorem*{definition}{Definition}
\newtheorem*{definitions}{Definitions}
\newtheorem*{conventions}{Conventions}
\newtheorem*{ack}{Acknowledgements}

\DeclareMathOperator{\des}{Des}
\DeclareMathOperator{\maj}{maj}

\DeclareMathOperator{\stp}{Step}
\DeclareMathOperator{\std}{std}

\DeclareMathOperator{\Tp}{\Theta(\pi^\prime)}
\DeclareMathOperator{\T}{\Theta(\pi)}

\newcommand{\makeset}[2]{ \{#1\;|\;#2\} }
\newcommand{\RL}[1]{ RL\left(#1\right) }
\newcommand{\LR}[1]{ LR\left(#1\right) }

\date{}

\begin{document}
\title{A refinement of Wilf-equivalence for patterns of length 4}
\author{Jonathan Bloom\\Dartmouth College\\Hanover, NH 03755\\(603) 646-2415\\{\tt jonathan.bloom@dartmouth.edu}}

\maketitle

\begin{abstract}
In their paper~\cite{DokosDwyer:Permutat12}, Dokos et al. conjecture that the major index statistic is equidistributed among 1423-avoiding, 2413-avoiding, and 3214-avoiding permutations. In this paper we confirm this conjecture by constructing two major index preserving bijections, $\Theta:S_n(1423)\to S_n(2413)$ and $\Omega:S_n(2314)\to S_n(2413)$.  In fact, we show that $\Theta$ (respectively, $\Omega$) preserves numerous other statistics including the descent set, right-to-left maxima (respectively, left-to-right minima), and a statistic we call steps.  Additionally, $\Theta$ (respectively, $\Omega$) fixes all permutations avoiding both $1423$ and $2413$ (respectively, $2314$ and $2413$).  
\end{abstract}

\section{Introduction}

We write a permutation $\pi\in S_n$ as $\pi_1\ldots \pi_n$ where $ |\pi|=n$ is called the length of $\pi$.  For any sequence of distinct integers $a_1\ldots a_n$ there is a unique permutation $\pi\in S_n$ with the defining property that  $a_i<a_j$ if and only if $\pi_i<\pi_j$, provided $i\neq j$.  We say that $a_1\ldots a_n$ is \emph{order isomorphic} to $\pi$ and write $\pi=\std(a_1\ldots a_n)$.  For example, $1342 = \std(2693)$.  Moreover, we say that $\pi$ \emph{avoids}  $\tau\in S_m$ if no subsequence of $\pi$ is order isomorphic to $\tau$.  We denote by $S_n(\tau)$ the set of all $\pi\in S_n$ that avoid $\tau$.  In this context we usually refer to $\tau$ as a \emph{pattern}.  For example, $5734612$ avoids $1423$ but does not avoid $2413$.     

If two patterns $\sigma, \tau\in S_m$ are such that $|S_n(\tau)| = |S_n(\sigma)|$ for all $n$, then we say that $\sigma$ is \emph{Wilf-equivalent} to $\tau$, and write $\sigma\sim\tau$.   For example, it is well known that $S_3$ has only one Wilf-equivalence class, while $S_4$ partitions into 3 classes~\cite{Bona:Combinat12}.

A refinement of Wilf-equivalence involves the idea of a \emph{permutation statistic}, which is defined to be a function $f:S_n\to T$, where $T$ is any fixed set.  We say $\sigma$ and $\tau$ are $f$-Wilf-equivalent if, for all $n$, there is some bijection $\Theta:S_n(\sigma) \to S_n(\tau)$ such that  $f(\pi) = f\left(\Theta(\pi)\right)$.  In other words, the $f$ statistic is equally distributed on the sets $S_n(\sigma)$ and $S_n(\tau)$.  In terms of the bijection $\Theta$, we say that $\Theta$ \emph{preserves} $f$.    This refinement has been heavily studied for patterns of length 3, for example see~\cite{BloomSaracino:On-bijec2009,BloomSaracino:Another-2010,DeutschRobertson:Refined-07,Elizalde:Fixed-po11,RobertsonSaracino:Refined-02}. A particularly nice (and nearly exhaustive) classification of Wilf-equivalent patterns of length 3 and permutation statistics is given by Claesson and Kitaev in~\cite{ClaessonKitaev:Classifi08}. On the other hand, little is known about permutation statistics and patterns of length 4 or greater.  To state a recent conjecture by Dokos et al.~\cite[Conjecture 2.8]{DokosDwyer:Permutat12} with regards to permutation statistics and patterns of length 4 we need the following definition. First, we say that $i$ is a \emph{descent} in $\pi$ if $\pi_i>\pi_{i+1}$ and we denote the set of all descents in $\pi$ by $\des \pi $.  The \emph{major index} is then defined to be $\maj(\pi) = \sum_{i\in \des \pi} i$.   With these definitions, Dokos et al. state the following conjecture:

\begin{conjecture}
The major index is equally distributed on the sets $S_n(2413)$, $S_n(1423)$, and $S_n(2314)$.
\end{conjecture}  

In a private communication, S. Elizalde conjectured further that the descent sets are equally distributed on the sets $S_n(1423)$ and $S_n(2413)$.  Additionally, B. Sagan conjectured that the positions of the $n$ and the $n-1$ are equally distributed on these two sets as well.   In this paper, we confirm Conjecture 1, as well as these stronger claims, by constructing explicit bijections $\Theta:S_n(1423)\to S_n(2413)$ and $\Omega:S_n(2314) \to S_n(2413)$ that preserve the descent set and, in the case of $\Theta$, preserve the position of the $n$ and the $n-1$.  In fact, we show that these maps simultaneously preserve several other permutation statistics. To describe these permutation statistics we require a few definitions.  

First, a \emph{right-to-left maximum} in $\pi$ is an index $i$ such that  $\pi_i>\pi_j$ provided $i<j$.  We denote by $\RL{\pi}$ the set of all right-to-left maxima in $\pi$.  Denoting $\RL{\pi} = \{i_1<\cdots <i_s\}$ we observe that 
$\pi_{i_1} = n$, $i_s = n$, $\pi_{i_1} > \cdots >\pi_{i_s}$, and $ \pi_{i_j}>\pi_k$, provided $i_j<k<i_{j+1}$.

Similarly, we define a  \emph{left-to-right minimum} in $\pi$ to be an index $i$ such that  $\pi_j>\pi_i$ provided $j<i$.  We denote by $\LR{\pi}$ the set of all left-to-right minima in $\pi$.  

The last permutation statistics we will need to consider are called  \emph{steps}, which are defined as follows.  A step in $\pi$ is an index $i$ such that $\pi_i-1 = \pi_{i+1}$.  The set of all steps in $\pi$ is denoted by $\stp \pi $. 

In the next section we prove the existence of a bijection $\Theta:S_n(1423) \to S_n(2413)$ that preserves descents, right-to-left-maxima, steps, and the position of the $n$ and $n-1$.  As a corollary, we show the existence of an $\Omega$ that preserves descents, left-to-right minima, steps, and the positions of the $1$ and the $2$.  

It should be mentioned that Stankova in~\cite{Stankova:Forbidde94} originally proved that $|S_n(1423)| = |S_n(2413)|$ by showing that their respective generating trees are isomorphic.  This result is important in the theory of pattern avoidance as it was one of the key results needed to finalize the Wilf-equivalences among all patterns of length 4.  In light of this, it is all the more interesting that an explicit multiple-statistic-preserving bijection exists between these two sets.  

Before closing this section let us define the interval $[a,b]=\{a,a+1,\ldots, b\}$, provided $a\leq b$. 

\section{The Bijections}
Before we begin the construction of our maps $\Theta$ and $\Omega$, let us state our main theorem and prove a corollary.  

\begin{theorem}[Main Theorem]\label{thm:main thm}
There is an explicit bijection $\Theta:S_n(1423)\to S_n(2413)$
such that $\Theta$ preserves descents, right-to-left maxima, steps, and the position of the $n$ and the $n-1$.  Additionally, if $\pi \in S_n(1423)\ \cap\ S_n(2413)$ then $\Theta(\pi) = \pi$.  
\end{theorem}

\begin{corollary}\label{cor:main cor}
There is an explicit bijection $\Omega:S_n(2314)\to S_n(2413)$
such that $\Omega$ preserves descents, left-to-right minima, steps, and the position of the $1$ and the $2$.  Additionally, if $\pi \in S_n(2314)\ \cap\ S_n(2413)$ then $\Omega(\pi) = \pi$.
\end{corollary}

In order to prove the corollary we will need a couple of standard definitions.   The \emph{complement} of $\pi\in S_n$ is defined to be $c(\pi) = (n+1- \pi_1)\ldots (n+1-\pi_n)$  and the \emph{reverse} of $\pi$ is appropriately given by $r(\pi) = \pi_n\pi_{n-1}\ldots \pi_1$.  It is clear that $r$ and $c$ are involutions such that $rc = cr$.

\begin{proof}[Proof of corollary]
First, observe that $rc(1423) = 2314$ and $rc(2413) = 2413$.  It follows that we may define $\Omega:S_n(2314) \to S_n(2413)$ as $\Omega = cr\circ\Theta \circ rc$.  Certainly $\Omega$ is bijective, since $r$, $c$, and $\Theta$ are bijections.   It only remains to show that $\Omega$ preserves the stated statistics. To this end fix  $\pi\in S_n(2314)$.   First, it is clear by the definitions involved that $\Omega$ preserves the positions of the $1$ and the $2$ and that  $\Omega(\pi) = \pi$ provided $\pi \in S_n(2314)\ \cap\ S_n(2413)$.  With respect to left-to-right minima, it follows from our Main Theorem that $\RL{rc (\pi)} = \RL{\Theta\circ rc( \pi) }$. As $\LR{cr(\pi)} = \makeset{n+1-i}{i\in \RL{\pi}}$, $c^{-1} = c$ and $r^{-1} = r$ it follows that 
$$\LR{\pi} = \LR{cr\circ rc (\pi)} = \LR{cr\circ\Theta\circ rc( \pi) } = \LR{\Omega(\pi)}.$$
The fact that $\Omega$ preserves descents and steps follows by a similar argument and the observations that $\des(cr(\pi) )= \makeset{n-i}{i\in \des \pi }$ and $\stp(cr(\pi)) = \makeset{n-i}{i\in \stp \pi }$.  The details are left to the reader.  
\end{proof}

To prove  Conjecture 1 it  now only remains to prove the Main Theorem; this occupies the remainder of the section. To motivate what follows we begin with an outline of the construction of  $\Theta$.   As  $\Theta$ will be constructed recursively we will need a way to decompose a 1423-avoiding permutation $\pi$ into two smaller 1423-avoiding permutations.  We will denote these smaller permutations as $\pi^{(1)}$ and $\pi^{(2)}$.  Having done so, we will then apply $\Theta$ to $\pi^{(1)}$ and $\pi^{(2)}$ to obtain two 2413-avoiding permutations, which we will denote $\sigma$ and $\alpha$ respectively.  Lastly, we will use an operation called inflation to combine $\sigma$ and $\alpha$ into one larger 2413-avoiding permutation denoted $\Theta(\pi)$.  

To begin let us define the inflation operation.   

\begin{definition}
Let $\sigma\in S_m$, $\alpha\in S_l$, and fix $1\leq a+1\leq m$.  We define $\sigma(\alpha,a+1)$ to be the permutation 
\[\hat{\sigma}_1\ldots \hat{\sigma}_{a}\ \hat{\alpha}_1\ldots \hat{\alpha}_l\ \hat{\sigma}_{a+2} \ldots \hat{\sigma}_m
\]
of length $l+m-1$, 
where $\hat{\alpha}_i= \alpha_i +\sigma_{a+1} -1$ and 
$ \hat{\sigma}_i = 
\begin{cases}
\sigma_i &\text{ if } \sigma_i<\sigma_{a+1}\\
\sigma_i + l-1 &\text{ otherwise.}
\end{cases}$
\end{definition}

In words, this operation essentially ``inflates" $\sigma_{a+1}$ with the permutation $\alpha$ and, in so doing, translates the other elements of $\sigma$ so that the result is a permutation. As a result we will refer to this operation as \emph{inflation}.    For example, if $\sigma = 316542$ and $\alpha = 531642$ then $\sigma(\alpha,4) = 3\ 1\ 11\ 9\ 7\ 5\ 10\ 8\ 6\ 4\ 2$.  An important property of inflation, which follows easily from the definition, is that
\begin{equation}\label{eqn:wreath prop}
\{\sigma'_{a+1},\ldots,\sigma'_{a+l}\} = [\sigma_{a+1}, \sigma_{a+1} + l -1],
\end{equation}
where $\sigma' = \sigma(\alpha,a+1)$.
To illustrate this property (and the operation of inflation) consider Figure~\ref{fig:wreath}.  This depiction of our example makes it clear that the set of values given by the elements in positions $4$ through $9$ are precisely the interval $[5,10]$

We now turn our attention to showing that $\sigma(\alpha,a+1)$ is 2413-avoiding, provided that $\sigma$ and $\alpha$ are both 2413-avoiding. 

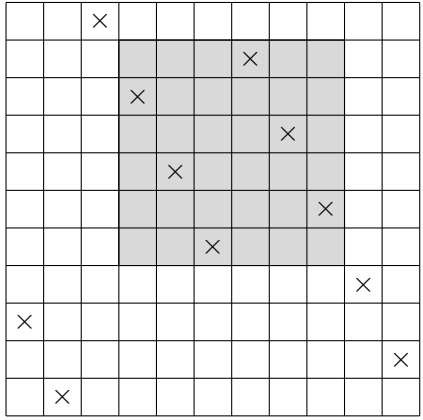
\begin{figure}[h!]
\begin{center}
\begin{tikzpicture}[scale=0.5]

\begin{scope}[shift={(0,0)}]
\draw[fill=gray!30] (3,4) rectangle (9,10);

\draw (0,0) -- (0,11);
\draw (1,0) -- (1,11);
\draw (2,0) -- (2,11);
\draw (3,0) -- (3,11);
\draw (4,0) -- (4,11);
\draw (5,0) -- (5,11);
\draw (6,0) -- (6,11);
\draw (7,0) -- (7,11);
\draw (8,0) -- (8,11);
\draw (9,0) -- (9,11);
\draw (10,0) -- (10,11);
\draw (11,0) -- (11,11);

\draw (0,0) -- (11,0);
\draw (0,1) -- (11,1);
\draw (0,2) -- (11,2);
\draw (0,3) -- (11,3);
\draw (0,4) -- (11,4);
\draw (0,5) -- (11,5);
\draw (0,6) -- (11,6);
\draw (0,7) -- (11,7);
\draw (0,8) -- (11,8);
\draw (0,9) -- (11,9);
\draw (0,10) -- (11,10);
\draw (0,11) -- (11,11);

\node at (0.5,2.5) {$\times$};
\node at (1.5,0.5) {$\times$};
\node at (2.5,10.5) {$\times$};
\node at (3.5,8.5) {$\times$};
\node at (4.5,6.5) {$\times$};
\node at (5.5,4.5) {$\times$};
\node at (6.5,9.5) {$\times$};
\node at (7.5,7.5) {$\times$};
\node at (8.5,5.5) {$\times$};
\node at (9.5,3.5) {$\times$};
\node at (10.5,1.5) {$\times$};
\end{scope}
\end{tikzpicture}
\caption{The picture of $\sigma(\alpha,4)$ where $\sigma = 316542$ and $\alpha = 531642$. The highlighted region represents the ``inflation"  of the element 5 in $\sigma$ by the permutation $\alpha$. }
\label{fig:wreath}
\end{center}
\end{figure}

\begin{lemma}\label{lem:2413 avoid}
If $\sigma\in S_m(2413)$ and $\alpha\in S_l(2413)$ then $\sigma(\alpha,a+1)\in S_{m+l-1}(2413)$.  
\end{lemma}

\begin{proof}
Assume for a contradiction that $\sigma' = \sigma(\alpha,a+1)$ contains an occurrence of 2413 and let  $\sigma'_z<\sigma'_x<\sigma'_w<\sigma'_y$ be such an occurrence where $x<y<z<w$.   If $|\{x,y,z,w\} \cap [a+1,a+l]|\leq 1$ it follows that $\sigma$ would contain an occurrence of 2413; so, we must have $|\{x,y,z,w\} \cap [a+1,a+l]|> 1$.  Now~\eqref{eqn:wreath prop} together with the fact that $|\{x,y,z,w\} \cap [a+1,a+l]|> 1$ implies that $\{x,y,z,w\} \subset [a+1,a+l]$.  Thus $\alpha$ must contain an occurrence of 2413.  As this contradicts the fact that $\alpha$ avoids 2413 we conclude that  $\sigma' \in S_{m+l-1}(2413)$.   
\end{proof}

We now turn our attention to the decomposition of $\pi$ we will need for our recursive construction of $\Theta$. 

\begin{definition}
Fix $\pi\in S_n$ so that $\pi \neq n(n-1)\ldots 1$ and let $\RL{\pi} =\{i_1<\cdots <i_r\}$.  Define $\phi(\pi) = (a,b)$ as follows.  First, let $b$ be the smallest right-to-left maximum so that $[b,n] \subset \RL{\pi}$. (In other words, $\pi_b$ is the smallest ascent top in $\pi$.) If $[b,n]= \RL{\pi}$, define $a=0$; otherwise, let $a=\max\left(\RL{\pi}\setminus [b,n]\right)$. 
\end{definition}

It immediately follows from this definition that $\RL{\pi} = \{i_1<\cdots< i_s<a\}\ \cup\ [b,n]$, for appropriately chosen indices $i_1\ldots i_s$.  

We are now ready to define the decomposition of $\pi$ needed. 

\begin{definitions}
Given $\pi\in S_n$ so that $\pi \neq n(n-1)\ldots 1$.  First define
$$\rho(\pi) = |\{i\in \RL{\pi}\ |\  i>b\mbox{ and } \pi_i>\chi(\pi)\}|\mbox{ and }\chi(\pi) = \max \{\pi_i\ |\ a<i<b\}.$$
Using these definitions, set  
$$\pi^{(1)} = \std\left(\pi_1\ldots \pi_a \pi_b \pi_{b+1}\ldots \pi_{b+\rho(\pi)}\right)\qquad\mbox{and}\qquad\pi^{(2)} = \std\left( \pi_{a+1}\ldots \pi_b \pi_{b+\rho(\pi)+1}\ldots \pi_n\right).$$
\end{definitions}

Before proceeding let us make a few remarks regarding this decomposition.  First, our assumption that $\pi \neq n(n-1)\ldots 1$ guarantees for us that $b-a\geq 2$.  As a result $\chi(\pi)>0$ and $|\pi^{(1)}| < n$ since the element $\pi_{a+1}$ is not included in its construction.  Second, the lengths of our two permutations are given by
\begin{equation}\label{eq:decomp size}
|\pi^{(1)}| = a+\rho(\pi)+1<n \qquad\mbox{and}\qquad|\pi^{(2)}| = n-a-\rho(\pi).
\end{equation}

Lastly, we observe that if $\pi\in S_n(1423)$, then 
\begin{equation}\label{eq:interval}
\{\pi_{a+1},\ldots,\pi_{b-1},\pi_{b+\rho(\pi)+1},\ldots, \pi_n\} = [1,\chi(\pi)] = [1, |\pi^{(2)}|-1].
\end{equation}
To see this note that if $\pi_i <\chi(\pi)$ and $i<a$ then $\pi_i\pi_a\chi(\pi)\pi_b$ would form a 1423-pattern.
An example of this decomposition which illustrates these aforementioned observations is in  Figure~\ref{fig:def pi}.

\begin{figure}[h!]
\begin{center}
\begin{tikzpicture}[scale=0.5]

\begin{scope}[shift={(0,0)}]
\draw[fill=gray!30] (6,0) rectangle (13,5);
\draw[fill=gray!30] (9,8) rectangle (10,9);

\foreach \i in {0,...,13}
	{
		\draw (\i,0) -- (\i,13);
		\draw (0,\i) -- (13,\i);
	}

\node at (0.5,9.5) {$\times$};
\node at (1.5,11.5) {$\times$};
\node at (2.5,12.5) {$\times$};
\node at (3.5,7.5) {$\times$};
\node at (4.5,5.5) {$\times$};
\node at (5.5,10.5) {$\times$};
\node at (6.5,4.5) {$\times$};
\node at (7.5,2.5) {$\times$};
\node at (8.5,0.5) {$\times$};
\node at (9.5,8.5) {$\times$};
\node at (10.5,6.5) {$\times$};
\node at (11.5,3.5) {$\times$};
\node at (12.5,1.5) {$\times$};
\end{scope}
\end{tikzpicture}
\caption{Here $\pi = 10\ 12\ 13\ 8\ 6\ 11\ 5\ 3\ 1\ 9\ 7\ 4\ 2$, and $\phi(\pi) = (6, 10)$, $\chi(\pi) = 5$, $\rho(\pi) = 1$.  Additionally, $\pi^{(1)} = \std(10\ 12\ 13\ 8\ 6\ 11\ 9\ 7)$ and $\pi^{(2)} = \std(531942)$.  To help visualize this decomposition we have shaded the elements of $\pi$ that become $\pi^{(2)}$}.
\label{fig:def pi}
\end{center}
\end{figure}
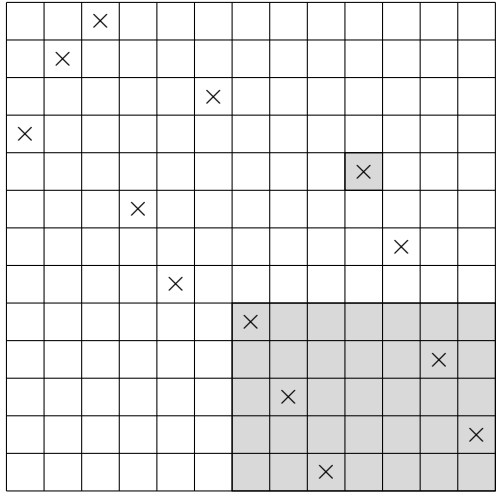

We are now in a position to define our map $\Theta$.

\begin{definition}
Fix $\pi\in S_n(1423)$ and set $\phi(\pi) = (a,b)$. We define the mapping $\Theta:S_n(1423) \to S_n(2413)$ as follows.       First, set $\Theta(1) = 1$.  Next, define $\Theta$ recursively via the following three cases:

\bigskip
{\tt Case 1:} $\stp \pi \ne \emptyset$

\medskip
First, let $k=\min\left(\stp \pi \right)$ so that $\pi'$ is the unique permutation such that $\pi = \pi'(21,k)$.  Now define $\T = \Tp(21,k)$.

\bigskip
{\tt Case 2:} $\stp \pi = \emptyset$ and $a=0$

\medskip
Denote by $\pi'$ the permutation obtained from $\pi$ by removing $n$. Define $\T$ as the permutation obtained from $\Tp$ by inserting $n$ in the $b$th position.   

\bigskip
{\tt Case 3:} $\stp \pi = \emptyset$ and $a\neq 0$

\medskip
Define $\T = \sigma(\alpha,a+1)$, where $\sigma = \Theta(\pi^{(1)})$ and $\alpha = \Theta(\pi^{(2)})$.
\end{definition}

It follows from the definition of $\Theta$ and a straightforward inductive argument that if $\pi \in S_n(1423)\ \cap\ S_n(2413)$ then $\Theta(\pi) = \pi$. 

Before proving that $\Theta$ is well-defined, bijective, and has the stated properties let us first remark as to the motivation of {\tt Case 1}.  Simply, the reason for this case is that without it, $\Theta$ will not be injective.  For example, if we ignored the step statistic, i.e., removed {\tt Case 1} from our definition of $\Theta$, then one can check that the permutations $52431$ and $53421$ would both map to $53421$.  The reader might find it helpful to keep this in mind while reading Lemma~\ref{lem:step wreath}, which is the crucial lemma for showing that $\Theta$ is injective.

To illustrate many of the arguments to follow let us now work out a complete example of the action of our map $\Theta$ on the permutation $\pi$ depicted in Figure~\ref{fig:def pi}.  As this value of $\pi$ falls in {\tt Case 3} we have

\begin{center}
\begin{tikzpicture}[scale=0.5]
\begin{scope}[shift={(0,0)}]
\node at (-1.3,4) {$\pi^{(1)}=$};
\foreach \i in {0,...,8}
	{
		\draw (\i,0) -- (\i,8);
		\draw (0,\i) -- (8,\i);
	}
\node at (0.5,4.5) {$\times$};
\node at (1.5,6.5) {$\times$};
\node at (2.5,7.5) {$\times$};
\node at (3.5,2.5) {$\times$};
\node at (4.5,0.5) {$\times$};
\node at (5.5,5.5) {$\times$};
\node at (6.5,3.5) {$\times$};
\node at (7.5,1.5) {$\times$};

\end{scope}

\begin{scope}[shift={(14,0)}]
\node at (-1.3,3) {$\pi^{(2)}=$};
\foreach \i in {0,...,6}
	{
		\draw (\i,0) -- (\i,6);
		\draw (0,\i) -- (6,\i);
	}
\node at (0.5,4.5) {$\times$};
\node at (1.5,2.5) {$\times$};
\node at (2.5,0.5) {$\times$};
\node at (3.5,5.5) {$\times$};
\node at (4.5,3.5) {$\times$};
\node at (5.5,1.5) {$\times$};
\end{scope}
\end{tikzpicture}\ .

\end{center}
Now let $\lambda = \pi^{(1)}$. Observing that $\phi(\lambda) = (3,6)$, we have 
\begin{center}
\begin{tikzpicture}[scale=0.5]
\begin{scope}[shift={(0,0)}]
\node at (-1.3,2.5) {$\lambda^{(1)}=$};
\foreach \i in {0,...,5}
	{
		\draw (\i,0) -- (\i,5);
		\draw (0,\i) -- (5,\i);
	}
\node at (0.5,1.5) {$\times$};
\node at (1.5,3.5) {$\times$};
\node at (2.5,4.5) {$\times$};
\node at (3.5,2.5) {$\times$};
\node at (4.5,0.5) {$\times$};

\end{scope}

\begin{scope}[shift={(14,0)}]
\node at (-1.3,2) {$\lambda^{(2)}=$};
\foreach \i in {0,...,4}
	{
		\draw (\i,0) -- (\i,4);
		\draw (0,\i) -- (4,\i);
	}
\node at (0.5,2.5) {$\times$};
\node at (1.5,0.5) {$\times$};
\node at (2.5,3.5) {$\times$};
\node at (3.5,1.5) {$\times$};

\end{scope}
\end{tikzpicture}\ .
\end{center}
Since $\lambda^{(j)}\in S_4(1423)\cap S_4(2413)$ it follows by our (yet unproven) claim that $\Theta(\lambda^{(j)})=\lambda^{(j)}$ for $j=1,2$.  Likewise, as $\pi^{(2)}$ is 2413-avoiding we also have $\alpha = \Theta(\pi^{(2)})= \pi^{(2)}$.  Therefore, we obtain

\begin{center}
\begin{tikzpicture}[scale=0.5]
\begin{scope}[shift={(0,0)}]
\draw[fill=gray!20] (3,2) rectangle (7,6);
\node at (-7,4) {$\sigma = \Theta(\pi^{(1)}) = \lambda^{(1)}(\lambda^{(2)},3+1)=$};
\foreach \i in {0,...,8}
	{
		\draw (\i,0) -- (\i,8);
		\draw (0,\i) -- (8,\i);
	}
\node at (0.5,1.5) {$\times$};
\node at (1.5,6.5) {$\times$};
\node at (2.5,7.5) {$\times$};

\node at (3.5,4.5) {$\times$};
\node at (4.5,2.5) {$\times$};
\node at (5.5,5.5) {$\times$};
\node at (6.5,3.5) {$\times$};

\node at (7.5,0.5) {$\times$};

\end{scope}
\end{tikzpicture}\ ,
\end{center}
and finally, as $\phi(\pi) = (6,10)$, 
\begin{center}
\begin{tikzpicture}[scale=0.5]

\node at (-5,6.5) {$\Theta(\pi) = \sigma(\alpha,6+1)=$};
\draw[fill=gray!20] (6,3) rectangle (12,9);
\begin{scope}[shift={(0,0)}]

\foreach \i in {0,...,13}
	{
		\draw (\i,0) -- (\i,13);
		\draw (0,\i) -- (13,\i);
	}

\node at (0.5,1.5) {$\times$};
\node at (1.5,11.5) {$\times$};
\node at (2.5,12.5) {$\times$};
\node at (3.5,9.5) {$\times$};
\node at (4.5,2.5) {$\times$};
\node at (5.5,10.5) {$\times$};

\node at (6.5,7.5) {$\times$};
\node at (7.5,5.5) {$\times$};
\node at (8.5,3.5) {$\times$};
\node at (9.5,8.5) {$\times$};
\node at (10.5,6.5) {$\times$};
\node at (11.5,4.5) {$\times$};

\node at (12.5,0.5) {$\times$};

\end{scope}
\end{tikzpicture}\ .

\end{center}

We now turn our attention to the proof of the Main Theorem.  We begin with several technical lemmas that deal with the relationship between inflation (respectively, decomposition) and the permutation statistics of interest, namely, right-left-maxima, descents, and steps.  For Lemmas~\ref{lem:RL pi} to~\ref{lem:recover from wreath} we establish the following conventions.

\begin{conventions}
Fix $\pi\in S_n(1423)$ so that $\pi \neq n(n-1)\ldots 1$ and let $\phi(\pi) = (a,b)$.  Additionally, we set 
\[
m = |\pi^{(1)}| = a+\rho(\pi)+1\qquad\mbox{and}\qquad l =|\pi^{(2)}| = n-a-\rho(\pi)
\] 
and let $\RL{\pi} = \{i_1<\cdots <i_s<a\}\ \cup \ [b,n]$.
\end{conventions}

\begin{lemma}\label{lem:RL pi} $\RL{\pi^{(1)}} = \{i_1<\cdots< i_s<a\}\ \cup [a+1,m]\quad\mbox{ and }\quad\RL{\pi^{(2)}} = [b-a, l]$.  
\end{lemma}
\begin{proof}
The first claim follows since the elements in the sequence $\pi_1\ldots \pi_a\pi_b\pi_{b+1}\ldots \pi_{b+\rho(\pi)}$ with the property that all the values to their right are smaller are
$\left\{\pi_{i_1},\ldots, \pi_{i_s},\pi_a,\pi_b,\pi_{b+1},\ldots, \pi_{b+\rho(\pi)}\right\}$.
Similarly, the second claim follows since the elements in the sequence $\pi_{a+1}\ldots \pi_{b}\pi_{b+\rho(\pi)+1}\ldots\pi_n$
with this same property are $\left\{\pi_b,\pi_{b+\rho(\pi)+1},\ldots, \pi_n\right\}$. 
\end{proof}

\begin{lemma}\label{lem:des pi} 
The relationship between the descents of $\pi$ and the descents of $\pi^{(1)}$ and $\pi^{(2)}$ is as follows:
$$\des \pi^{(1)} = \makeset{i\in\des \pi }{1\leq i\leq a}\ \cup [a+1,m-1]$$
and 
$$\des \pi^{(2)} = \makeset{i-a}{i\in\des\pi \mbox{ and } a< i< b}\ \cup\  [b-a, l-1].$$
\end{lemma}
\begin{proof}
This proof follows easily by the definition of $\pi^{(1)}$ and $\pi^{(2)}$, Lemma~\ref{lem:RL pi}, and the fact that if $i\in \RL{\pi^{(j)}}$ and $i<|\pi^{(j)}|$ then $i\in \des \pi^{(j)}$, where $j=1,2$.
\end{proof}

\begin{lemma}\label{lem:step pi}
If $\stp \pi = \emptyset$, then $\stp\pi^{(1)} \subset\{a\}$ and $\stp\pi^{(2)} = \emptyset$.
\end{lemma}

\begin{proof}
This readily follows from Equation~\eqref{eq:interval}.  Note that $\stp{\pi^{(1)}} = \{a\}$ if and only if\\ $\pi_a -1= \pi_b$.
\end{proof}

\begin{conventions}
For Lemmas~\ref{lem:RL wreath} to \ref{lem:recover from wreath}  fix  $\sigma\in S_m$,  and $\alpha\in S_l$ where $m+l-1 = n$.  Set 
$$\RL{\sigma} = \{i_1<\cdots< i_s<a\}\ \cup\ [a+1,m]$$ and $\RL{\alpha} = [b-a,l]$. Lastly, let $\sigma' = \sigma(\alpha,a+1)$. 
\end{conventions}

\begin{lemma}\label{lem:RL wreath}$\RL{\sigma'} = \{i_1,\cdots, i_s, a\}\ \cup\ [b,n]$.
\end{lemma}

\begin{proof}
It is clear from the definition of inflation that 
$$\makeset{i\in\RL{\sigma'}}{i\leq a}=\{i_1<\cdots< i_s<a\}\qquad\mbox{and}\qquad\makeset{i\in\RL{\sigma'}}{ a+l<i} = [a+l+1, m+l-1].$$
Lastly, since we are inflating the $(a+1)$st element of $\sigma$ where $a+1\in \RL{\sigma}$ it follows that 
$$ \makeset{i\in\RL{\sigma'}}{ a+1\leq i\leq a+l} = [b, a+l],$$
which completes the proof.  
\end{proof}

\begin{lemma}\label{lem:des wreath} If  $\des \sigma = \makeset{i}{i\in \des \pi, 1\leq i\leq a}\ \cup \ [a+1,m-1]$
and $$\des \alpha = \makeset{i-a}{i\in \des \pi, a< i< b}\ \cup \ [b-a, l-1],$$
then  $\des \sigma' = \des \pi $.
\end{lemma}

\begin{proof}
It follows directly from the definitions that 
$$\des \sigma' = \makeset{i}{i\in \des \pi, 1\leq i\leq a}\ \cup\ \des \alpha +a\ \cup\  [a+1,m-1]+(l-1).$$  Since $l = n-a-\rho(\pi)$ and $m = a+\rho(\pi) +1$ it follows that 
$$\des \sigma' = \makeset{i}{i\in \des \pi , 1\leq i< b}\ \cup\  [b,n-1] = \des \pi.$$
\end{proof}

\begin{lemma}\label{lem:step wreath}
Assume $\stp \alpha = \emptyset$ and $\stp \sigma \subset\{a\}$. Then $\stp \sigma' = \emptyset$ provided $\alpha_1<l$
\end{lemma}
\begin{proof}
It is clear from the definition of inflation that we must have $\stp \sigma' \subset\{a,a+l\}$. Now, the only way we could have $a\in \stp \sigma'$ is for $\alpha_1=l$, which is prohibited, so $a\notin \stp \sigma'$.  Similarly, the only way $a+l$ could be a step in $\sigma'$ is for $a+1\in \stp \sigma$.  As this is prohibited we conclude that $\stp \sigma' = \emptyset$.  
\end{proof}

\begin{lemma}\label{lem:recover from wreath}
Assume that $\stp \sigma\subset\{a\}$.  Then $a+l$ is the largest index such that 
$$\{\sigma'_{a+1}, \ldots, \sigma'_{a+l}\}$$
is an interval.  An immediate consequence of this is that $\alpha$ and $\sigma$ may be recovered from $\sigma'$ and $a$.  
\end{lemma}
\begin{proof}
Let $k$ be the largest index such that $\{\sigma'_{a+1}, \ldots, \sigma'_k\}$ is an interval.  By Equation~\eqref{eqn:wreath prop} we know that 
$$\{\sigma'_{a+1},\ldots,\sigma'_{a+l}\} = [\sigma_{a+1}, \sigma_{a+1} + l -1],$$ 
and hence $a+l\leq k$.   If $a+1=m$ we are done.  Otherwise, as $[a+1,m]\in\RL{\sigma}$ we must have $\sigma_{a+1}>\sigma_{a+2}>\cdots$.  Moreover, as $\stp(\sigma)\subset\{a\}$ it follows that $\sigma_{a+1}-1$ must be in some position to the left of position $a+1$.  It now readily follows that $k\leq a+l$, completing the proof.  

\end{proof}

With the technical lemmas proved we now turn our attention to the proof of the Main Theorem.  

\begin{proof}[Proof of Main Theorem]

It is clear from the definition of $\Theta$ that when $n=1$ the map is well defined and the theorem holds.  As $\Theta$ is defined recursively we will proceed by induction on $n$, and fix $\pi\in S_n$.  To simplify the following arguments it will be useful to recall the fact that  $|S_n(1423)| = |S_n(2413)|$, (see~\cite{Bona:Combinat12}).  Therefore, to show that $\Theta$ is bijective it will suffice to establish the injectivity of $\Theta$.  We proceed by considering, independently, the three defining of $\Theta$.  In each case we will also show that $\Theta$ preserves Steps and that $\phi(\T) = \phi(\pi)$.  Therefore we may conclude that if $\Theta$ is injective in each case then it must also be, globally, injective as well.  To this end, observe that once we show $\RL{\pi} = \RL{\T}$ it immediately follows that 
\begin{equation}\label{eq:same phi}
\phi(\pi) = (a,b) = \phi(\T).
\end{equation}

\bigskip
{\tt Case 1:} $\stp \pi \neq \emptyset$
\bigskip 

We first check that $\Theta$ is well defined in this case.  Recall from the definition of $\Theta$ that $\T =\Tp(21,k)$ where $\pi'\in S_{n-1}$ is the unique permutation such that $\pi = \pi'(21,k)$ and $k=\min\left(\stp \pi\right)$.  By our induction hypothesis we may assume that $\Tp\in S_{n-1}(2413)$.  Since $\stp 2413 = \emptyset$ it follows that  $\T\in S_n(2413)$.  In this case, it readily follows that $\Theta$ preserves right-to-left maxima, descents, steps, and the position of $n$ and $n-1$. Likewise, it follows that $\Theta(\pi) = \pi$ provided $\pi\in S_n(1423)\ \cap\ S_n(2413)$.

The fact that we may recover $\pi$ from $\T$, in this case, easily follows by our induction hypothesis and the fact that $\stp \pi = \stp \T$.

\bigskip
{\tt Case 2:} $\stp \pi = \emptyset$ and $a= 0$
\bigskip

In this case $\RL{\pi} = [b,n]$ and, hence, $\pi_b= n$.  Denoting by $\pi'$ the permutation obtained by deleting $n$ from $\pi$, we see that
$$[b,n-1]\subset \RL{\pi'} = \RL{\Tp},$$
where the equality follows by our induction hypothesis.    Since $\T$ is obtained from $\Tp$ by inserting an $n$ in the $b$th position it is now clear that  $\RL{\T}=[b,n]=\RL{\pi}$.  

The fact that  $\Theta$, in this case, preserves descents, the position of $n$ and $n-1$, and that $\Theta(\pi) = \pi$ provided $\pi\in S_n(1423)\ \cap\ S_n(2413)$ is straightforward.

Now consider the step statistic.  Certainly, $\stp \Tp = \stp \pi' = \emptyset$.  Therefore we only need to show that $n-1$ is not in the $(b+1)$st position of $\Tp$.  As this is the largest element, this could only occur if $\pi'_{b+1} = n-1$.  This is impossible since $\pi_b = n$ and $b\notin \stp \pi$.

To see that $\T$ is 2413-avoiding, observe that any occurrence of 2413 in $\T$ would have positions  $x<b<y<z$ since $\Tp \in S_{n-1}(2413)$.  This is impossible as we have established that $[b,n]=\RL{\T}$, which precludes $\T_y<\T_z$.  

In this case, the fact that we may recover $\pi$ from $\T$, follows immediately from Equation~\eqref{eq:same phi} and our induction hypothesis.

\bigskip
{\tt Case 3:} $\stp \pi = \emptyset$ and $a>0$
\bigskip

First consider the permutations $\pi^{(1)}\in S_m(1423)$ and $\pi^{(2)}\in S_l(1423)$ where $m = a +\rho(\pi)+1$ and $l = n-a - \rho(\pi)$. As $a>0$ we see that $l<n$. By Equation~\eqref{eq:decomp size} we are also guaranteed, since $\pi \neq n(n-1)\ldots 1$ that $m<n$.   Consequently, we may inductively define $\T = \sigma(\alpha,a+1)$ in this case where
\[
\sigma = \Theta(\pi^{(1)})\in S_m(2413) \qquad\mbox{and}\qquad \alpha = \Theta(\pi^{(2)})\in S_l(2413).
\]
By Lemma~\ref{lem:2413 avoid} we further know that $\T \in S_n(2413)$.  

To see that $\Theta$ preserves the required statistics, let $\RL{\pi} = \{i_1<\cdots <i_s< a\}\ \cup\ [b,n]$.  Now Lemma~\ref{lem:RL pi}, our induction hypothesis, and Lemma~\ref{lem:RL wreath} together imply that  $\RL{\T} =  \{i_1,\ldots , i_s,a\}\ \cup\ [b,n] = \RL{\pi}$.  

Likewise, Lemmas~\ref{lem:des pi} and ~\ref{lem:des wreath} together with our induction hypothesis imply that $\des \T=  \des \pi$.  

By Equation~\eqref{eq:interval} it follows that $\pi^{(2)}_1<l$ and since the position of the $l$ is preserved under $\Theta$ we must also have $\alpha_1 < l$. It now follows by Lemmas~\ref{lem:step pi} and~\ref{lem:step wreath} that $\stp \T= \emptyset =  \stp \pi$.

Now consider the position of $n$ and $n-1$.  As $\Theta$ preserves the right-to-left maxima it must preserve the position of $n$.    Now let $i$ be such that $\pi_i = n-1$. If $i\leq a$ then $\pi^{(1)}_i$ is the second largest element in $\pi^{(1)}$ and by induction $\sigma_i$ is the second largest element in $\sigma$. It now follows that $\Tp_i = n-1$ as well.  On the other hand if $a<i$ then we must have $i=b$.  In this case $\pi^{(2)}_{b-a}$ is the largest element in $\pi^{(2)}$ and $\pi^{(1)}_{a+1}$ is the second largest element in $\pi^{(1)}$.  By induction and the definition of inflation it follows that $\T_i=\sigma(\alpha,a+1)_i = n-1$ in this case too. 

Now assume $\pi \in S_n(1423)\ \cap\ S_n(2413)$.  In this case $\pi_i>\pi_b$ for all $i<a$.  For otherwise, $\pi_i \pi_a \pi_k \pi_b$ is either a 1423-pattern or a 2413-pattern, where $\pi_k = \chi(\pi)$.  Moreover, since $\stp \pi  = \emptyset$, we must have $\rho(\pi) = 0$.  Together, these two observations imply that
\[
\pi^{(1)} = \std\left(\pi_1\ldots \pi_a \pi_b \right)\qquad\mbox{and}\qquad\pi^{(2)} =  \pi_{a+1}\ldots \pi_b \pi_{b+1}\ldots \pi_n,
\]
where $\pi^{(1)}_{a+1} = 1$. By induction, since  $\pi^{(j)}\in S_n(1423)\ \cap\ S_n(2413)$, we see that $\sigma = \pi^{(1)}$ and $\alpha = \pi^{(2)}$.  As $\sigma_{a+1} = 1$, it now follows that $\T= \sigma(\alpha, a+1) =\pi$.

We now turn our attention to showing that we may recover $\pi$ from $\T$ in this case.  It follows by Lemmas~\ref{lem:step pi} and~\ref{lem:recover from wreath} and our induction hypothesis that we may recover $\sigma$ and $\alpha$  from $\T$ and hence $\pi^{(1)}$ and $\pi^{(2)}$. We now show that  $\pi$ may be recovered from $\pi^{(1)}$, $\pi^{(2)}$, and $(a,b)$. By Equation~\eqref{eq:interval} it follows that the permutation obtained by deleting the largest value in $\pi^{(2)}$ is 
\[
\pi_{a+1}\ldots \pi_{b-1} \pi_{b+\rho(\pi)+1}\ldots \pi_n.
\]
Additionally, it also follows that $\chi(\pi) = l-1$.  Moreover, observe that adding $l-1$ to each element of $\pi^{(1)}$ yields the sequence $\pi_1\ldots \pi_a \pi_b \pi_{b+1}\ldots \pi_{b+\rho(\pi)}$.  This completes our proof.    
\end{proof}

\begin{ack}
The author is grateful to  Sergi Elizalde for recommending this problem.  Additionally, the author is grateful to Dan Saracino for help clarifying the exposition and to  Alex Burstein for pointing out the connection of this work to~\cite{Stankova:Forbidde94}. 
\end{ack}

\bibliography{/Users/jonathan/GoogleDrive/Papers/mybib}

\begin{thebibliography}{1}

\bibitem{BloomSaracino:On-bijec2009}
{\sc J.~Bloom and D.~Saracino}, {\em On bijections for pattern-avoiding
  permutations}, J. Combin. Theory Ser. A, 116 (2009), pp.~1271--1284.

\bibitem{BloomSaracino:Another-2010}
\leavevmode\vrule height 2pt depth -1.6pt width 23pt, {\em Another look at
  bijections for pattern-avoiding permutations}, Adv. in Appl. Math., 45
  (2010), pp.~395--409.

\bibitem{Bona:Combinat12}
{\sc M.~B{\'o}na}, {\em Combinatorics of permutations}, Discrete Mathematics
  and its Applications (Boca Raton), CRC Press, Boca Raton, FL, second~ed.,
  2012.
\newblock With a foreword by Richard Stanley.

\bibitem{ClaessonKitaev:Classifi08}
{\sc A.~Claesson and S.~Kitaev}, {\em Classification of bijections between 321-
  and 132-avoiding permutations}, S\'em. Lothar. Combin., 60 (2008/09),
  pp.~Art. B60d, 30.

\bibitem{DeutschRobertson:Refined-07}
{\sc E.~Deutsch, A.~Robertson, and D.~Saracino}, {\em Refined restricted
  involutions}, European J. Combin., 28 (2007), pp.~481--498.

\bibitem{DokosDwyer:Permutat12}
{\sc T.~Dokos, T.~Dwyer, B.~Johnson, B.~Sagan, and K.~Selsor}, {\em Permutation
  patterns and statistics}, Discrete Math., 312 (2012), pp.~2760--2775.

\bibitem{Elizalde:Fixed-po11}
{\sc S.~Elizalde}, {\em Fixed points and excedances in restricted
  permutations}, Electron. J. Combin., 18 (2011), pp.~Paper 29, 17.

\bibitem{RobertsonSaracino:Refined-02}
{\sc A.~Robertson, D.~Saracino, and D.~Zeilberger}, {\em Refined restricted
  permutations}, Ann. Comb., 6 (2002), pp.~427--444.

\bibitem{Stankova:Forbidde94}
{\sc Z.~E. Stankova}, {\em Forbidden subsequences}, Discrete Math., 132 (1994),
  pp.~291--316.

\end{thebibliography}
\bibliographystyle{siam}

\end{document}